\documentclass[12pt]{article}
\usepackage{amsmath,amssymb,amsthm,graphicx}
\newcommand{\re}{\mathbb{R}}
\newcommand{\co}{\mathbb{C}}

\newcommand{\na}{\mathbb{N}}

\newcommand{\cc}{\mathcal{C}}
\newcommand{\LL}{\mathcal{L}}
\newcommand{\z}{\bar z}

\newcommand{\tz}{\tilde z}

\long\def\symbolfootnote[#1]#2{\begingroup%
\def\thefootnote{\fnsymbol{footnote}}\footnote[#1]{#2}\endgroup}

\newtheorem{thm}{Theorem}[section]

\newtheorem{lem}[thm]{Lemma}

\theoremstyle{definition}

\newtheorem{example}[thm]{Example}

\theoremstyle{remark}

\newtheorem{rem}[thm]{Remark}

\title{Smooth counterexamples to strong unique continuation for a
  Beltrami system in $\mathbb{C}^2$}

\author{Adam Coffman \\ Yifei Pan}

\begin{document}

\maketitle

\begin{abstract}
We construct an example of a smooth map $\mathbb{C}\to\mathbb{C}^2$
which vanishes to infinite order at the origin, and such that the
ratio of the norm of the $\bar z$ derivative to the norm of the $z$
derivative also vanishes to infinite order.  This gives a
counterexample to strong unique continuation for a vector valued
analogue of the Beltrami equation.
\end{abstract}
  
\section{Introduction}\label{sec0}\symbolfootnote[0]{MSC 2010: 35A02 (Primary); 32W50, 35J46, 35R45 (Secondary)}

We will construct an example of a smooth function ${\bf
  u}:\co\to\co^2$ which has an isolated zero of infinite order at the
origin ($\|z^{-k}{\bf u}(z)\|\to0$ as $z\to0$ for all $k\ge0$), and
where the ratio of norms of derivatives ${\|{\bf u}_{\z}\|}/{\|{\bf
    u}_{z}\|}$ is small, also vanishing to infinite order at $z=0$.
This behavior is obviously different from that of a map ${\bf u}$ with
${\bf u}_{\z}\equiv{\bf0}$, which would be holomorphic and could not
have an isolated zero of infinite order.  This vector valued case is
also different from the complex scalar case, where solutions
$u:\co\to\co$ of the well-known Beltrami equation $u_{\z}=a(z)u_z$,
for small $a(z)$, also cannot vanish to infinite order at an isolated
zero (\cite{b}, \cite{ch}, \cite{aim}, \cite{rosay}).

More precisely, we will show in Section \ref{sec3} that in a
neighborhood of the origin, ${\bf u}(z)$ is a solution of a
Beltrami-type system of differential equations, which is linear,
elliptic, and has continuous coefficients very close to those of the
Cauchy-Riemann system, but does not have the property of strong unique
continuation.

The construction was motivated by an example of Rosay (\cite{rosay})
and questions posed by \cite{is}, who were considering the unique
continuation problem for systems of equations from almost complex
geometry.
  
In Section \ref{sec1}, we develop a general framework for constructing
smooth maps $\co\to\co^2$ vanishing to infinite order.  In Section
\ref{sec2}, we present both Rosay's example and our new example.  In
Section \ref{sec4} we state some open questions.

\section{General Setup}\label{sec1}

\subsection{Annular cutoff functions}

Start with a real valued function $s(x)$ which is smooth on $\re$,
with $s\equiv0$ on $[0,\frac14]$, $s$ increasing on
$[\frac14,\frac34]$, $s(\frac12)=\frac12$, $s^\prime(\frac12)=2$,
$s^{\prime\prime}(\frac12)=0$, and $s\equiv1$ on $[\frac34,1]$.

For $r_1>0$ and two parameters $0<r<r_1$ and $0<\Delta r<r_1-r$,
denote the annulus $A_{r,\Delta r}=\{z=x+iy\in\co:r\le|z|\le r+\Delta
r\}$ (contained in the disk $D_{r_1}$), and define a family of
functions $\chi_{r,\Delta r}:A_{r,\Delta r}\to\re$ by the formula
$\chi_{r,\Delta r}(z)=s\left(\frac{|z|-r}{\Delta r}\right)$.  At a
particular point $\tz=(\tilde x,\tilde y)\in A_{r,\Delta r}$,
\begin{eqnarray*}
  \frac{\partial}{\partial x}\left[\chi_{r,\Delta
    r}(x,y)\right]_{(\tilde x,\tilde y)}&=&\frac\partial{\partial
    x}\left[s\left(\frac{\sqrt{x^2+y^2}-r}{\Delta
    r}\right)\right]_{(\tilde x,\tilde y)}\\
    &=&s^\prime\left(\frac{\sqrt{\tilde x^2+\tilde y^2}-r}{\Delta
    r}\right)\cdot\frac{\tilde x}{\sqrt{\tilde x^2+\tilde
    y^2}}\cdot\frac1{\Delta r}\\
    &=&s^\prime\left(\frac{|\tz|-r}{\Delta r}\right)\cdot\frac{\tilde
    x}{|\tz|\Delta r}.
\end{eqnarray*}
The $y$ derivative is similar, and the $z$, $\z$ derivatives are
complex linear combinations.  In particular,
\begin{eqnarray}
  \frac{\partial}{\partial\z}\chi_{r,\Delta
    r}(z)&=&\overline{\frac{\partial}{\partial z}\chi_{r,\Delta
    r}(z)}=s^\prime\left(\frac{|z|-r}{\Delta
    r}\right)\cdot\frac{x+iy}{2|z|\Delta
    r}\label{eq0}\\ \implies\left|\frac{\partial}{\partial\z}\chi_{r,\Delta
    r}(z)\right|&=&\left|\frac{\partial}{\partial z}\chi_{r,\Delta
    r}(z)\right|\le\frac{m_{01}}{\Delta r}\label{eq-1}
\end{eqnarray}
for some constant $m_{01}>0$ not depending on $r_1$, $r$ or $\Delta
r$.

For higher derivatives of $\chi_{r,\Delta r}$, the following Lemma is
a simplified version of the Fa\`a di Bruno formula for derivatives of
composites.
\begin{lem}\label{lem1.1}
  For $k\ge0$, there exist polynomials $p_{abc}(x_1,x_2,x_3)$,
  $q_{abc}(x_1,x_2,x_3)$ indexed by $a,b,c\ge0$, $a+b=k$, $c\le k$,
  with constant complex coefficients (not depending on $r_1$, $r$,
  $\Delta r$, or $s$), so that
  $$\frac{\partial^a}{\partial
    z^a}\frac{\partial^b}{\partial\z^b}\chi_{r,\Delta
    r}(z)=\sum_{c=0}^ks^{(c)}\left(\frac{|z|-r}{\Delta
    r}\right)\cdot\frac{p_{abc}(z,\z,\Delta
    r)+|z|q_{abc}(z,\z,\Delta r)}{|z|^{2k}(\Delta r)^k}.$$
\end{lem}
\begin{proof}
  The $k=0$ case is trivial and the $k=1$ case is stated above.  We
  record the second derivatives:
\begin{eqnarray}
  \frac{\partial^2}{\partial\z^2}\chi_{r,\Delta r}(z)&=&\overline{\frac{\partial^2}{\partial z^2}\chi_{r,\Delta r}(z)}\nonumber\\
  &=&s^{\prime\prime}\left(\frac{|z|-r}{\Delta r}\right)\frac{z^2}{4|z|^2(\Delta r)^2}+s^\prime\left(\frac{|z|-r}{\Delta r}\right)\frac{-z^2}{4|z|^3\Delta r},\label{eq-2}\\
  \frac{\partial^2}{\partial z\partial\z}\chi_{r,\Delta r}(z)&=&s^{\prime\prime}\left(\frac{|z|-r}{\Delta r}\right)\frac1{4(\Delta r)^2}+s^\prime\left(\frac{|z|-r}{\Delta r}\right)\frac1{4|z|\Delta r}.\nonumber
\end{eqnarray}
  The proof for all larger $k$ is by induction on $k$; the calculation
  is straightforward and omitted here.
\end{proof}
It follows as a consequence of the Lemma that there are positive
constants $m_{ab}$ (indexed by $a,b\ge0$, $a+b=k$, and depending on
the choices of $s$ and $r_1$, but not depending on $r$, $\Delta r$) so
that
\begin{eqnarray*}
  \left|\frac{\partial^a}{\partial
  z^a}\frac{\partial^b}{\partial\z^b}\chi_{r,\Delta
  r}(z)\right|&\le&\frac{m_{ab}}{|z|^{2k}(\Delta r)^{k}}\\
  \implies\max_{z\in A_{r,\Delta r}}\left|\frac{\partial^a}{\partial
  z^a}\frac{\partial^b}{\partial\z^b}\chi_{r,\Delta
  r}(z)\right|&\le&\frac{m_{ab}}{r^{2k}(\Delta r)^{k}}.
\end{eqnarray*}
In various cases, in particular $k=1$ as in (\ref{eq-1}), the $r^{2k}$
can be improved (with a smaller exponent), but it is good enough to
use later in Lemma \ref{lem1.2}.

\subsection{The basic construction of the examples}

Let $r_n$ be a real sequence decreasing with limit $=0$.  Denote
$\Delta r_n=r_n-r_{n+1}$.

Let $A_n$ denote the closed annulus $$A_n=A_{r_{n+1},\Delta
  r_n}=\{z\in\co:r_{n+1}\le|z|\le r_n\},$$ so the union is a disk:
$D_{r_1}=(\cup A_n)\cup\{0\}$.  The annular cutoff functions can be
indexed by $n$: $\chi_n=\chi_{r_{n+1},\Delta r_n}:A_n\to\re$.

For $n\in\na$, let $p(n)$ be an increasing positive integer sequence.
Let $F(n)$ be a positive real valued sequence.  Define a function
${\bf u}:D_{r_1}\to\co^2$ by ${\bf u}(0)=\left[\begin{array}{c}0\\0\end{array}\right]$ and on the annulus $A_n$, for even
$n$:
\begin{eqnarray}
  {\bf u}(z)&=&\left[\begin{array}{c}u^1(z)\\u^2(z)\end{array}\right],\nonumber\\
  u^1(z)&=&F(n)z^{p(n)},\label{eq38}\\
  u^2(z)&=&\chi_n(z)F(n-1)z^{p(n-1)}+(1-\chi_n(z))F(n+1)z^{p(n+1)}.\label{eq39}
\end{eqnarray}
For odd $n$, switch the formulas for $u^1$, $u^2$.

So far, for any $s$, $r_n$, $p$, $F$, the function ${\bf u}$ is smooth
on $D_{r_1}\setminus\{0\}$.  We also have that ${\bf u}$ and ${\bf
u}_z$ have non-zero value at every point of $D_{r_1}\setminus\{0\}$;
for $n$ even (and switching indices if $n$ is odd):
\begin{equation}\label{eq46}
  \|{\bf u}_z\|\ge\left|\frac{\partial}{\partial
    z}u^1\right|=F(n)p(n)|z|^{p(n)-1}.
\end{equation}

\subsection{Smoothness at the origin}

An important property of the examples ${\bf u}:D_{r_1}\to\co^2$ we
want to construct is that they are smooth at (and near) the origin.
It is not enough to check only that the components vanish to infinite
order.  In general, as easily constructed examples would show, for
functions $f:\re^N\to\re^M$, $f$ can vanish to infinite order at the
origin: $$\displaystyle{\lim_{\vec x\to\vec0}\frac{f(\vec x)}{\|\vec
x\|^k}=\vec0}$$ for all whole numbers $k$, but need not be smooth.
Our approach to proving smoothness of our examples will be to show
$u^1$, $u^2$, and all their higher partial derivatives approach $0$;
this implies vanishing to infinite order, as in the following Lemma.

\begin{lem}\label{thm3.10}
  Given $f:\re^2\setminus{\vec0}\to\re^1$, suppose $f$ is smooth and
  for each $j,k=0,1,2,3,\ldots$,
  $$\lim_{(x,y)\to{\vec0}}\frac{\partial^{k+j}f(x,y)}{\partial
  x^k\partial y^j}=0.$$ Then extending $f$ so that $f(0,0)=0$ defines
  a smooth function on $\re^2$ that vanishes to infinite order at the
  origin.
\end{lem}
\begin{proof}
  $f$ is continuous at $\vec0$ by hypothesis ($j=k=0$).  To show $f$ is
  smooth, we only need to show every partial derivative of order
  $\ell$ of $f$ exists at $\vec0$, and has value $0$; then it follows
  that for $k+j=\ell$, $\frac{\partial^{k+j}f(x,y)}{\partial
  x^k\partial y^j}$ is continuous at $\vec0$.

  The proof is by induction on $\ell$; suppose for any non-commutative
  word $x^{k_1}y^{j_1}\cdots x^{k_a}y^{j_b}$ with $k_1+\cdots
  k_a+j_1+\ldots+j_b=k+j=\ell$, $\frac{\partial^{\ell}}{\partial
  x^{k_1}y^{j_1}\cdots x^{k_a}y^{j_b}}f(x,y)$ exists at $\vec0$ and has
  value $0$.  Then, the $x$-derivative at the origin is (with the
  $y$-derivative being similar):
  \begin{eqnarray*}
    &&\left.\frac{\partial}{\partial
    x}\left(\frac{\partial^{\ell}f(x,y)}{\partial x^{k_1}y^{j_1}\cdots
    x^{k_a}y^{j_b}}\right)\right]_{(0,0)}\\
    &=&\lim_{t\to0}\frac{\frac{\partial^{\ell}}{\partial
    x^{k_1}y^{j_1}\cdots
    x^{k_a}y^{j_b}}f(0+t,0)-\frac{\partial^{\ell}}{\partial
    x^{k_1}y^{j_1}\cdots x^{k_a}y^{j_b}}f(0,0)}{t}\\
    &=&\lim_{t\to0}\frac{\frac{\partial^{\ell}}{\partial
    x^{k}y^{j}}f(t,0)-0}{t}.
  \end{eqnarray*}
  Let $g(t)=\frac{\partial^{\ell}}{\partial x^{k}y^{j}}f(t,0)$; then $\displaystyle{\lim_{t\to0}g(t)=0}$ by hypothesis, and $g^\prime(t)=\frac{\partial^{\ell+1}}{\partial x^{k+1}y^{j}}f(t,0)$.  L'H\^opital's Rule applies to the above limit: 
\begin{eqnarray*}
  \lim_{t\to0}\frac{g(t)}{t}&=&\lim_{t\to0}\frac{g^\prime(t)}{1}=\lim_{t\to0}\frac{\partial^{\ell+1}}{\partial x^{k+1}y^{j}}f(t,0)=0.
\end{eqnarray*}
  The property of vanishing to infinite order follows from Taylor
  approximation at the origin.
\end{proof}

For our examples ${\bf u}$, we want to choose $r_n$, $p$, and $F$, so
that ${\bf u}$ is smooth and vanishes to infinite order at $0$.  The
following criterion for smoothness will be verified for both the
Examples in Section \ref{sec2}.
\begin{lem}\label{lem1.2}
  If $\frac{(\Delta r_n/r_n)}{(\Delta r_{n+2}/r_{n+2})}$ is a bounded
  sequence and, for each integer $k\ge0$,
  \begin{equation}\label{eq42}
    \lim_{n\to\infty}\frac{F(n+1)(p(n+1))^kr_n^{p(n+1)-4k}}{(\Delta
r_n/r_n)^k}=0,
  \end{equation}
  then ${\bf u}$ is smooth and vanishes to infinite order at the
origin.
\end{lem}
\begin{proof}
  From Lemma \ref{thm3.10} and the construction of ${\bf u}$, it is
  enough to show, for any non-negative integers $a$, $b$, with
  $a+b=k$, that
$$\displaystyle{\max_{z\in A_n}\left|\left(\frac{\partial}{\partial
z}\right)^a\left(\frac{\partial}{\partial\z}\right)^bu^1\right|}, \ \
\ \displaystyle{\max_{z\in A_n}\left|\left(\frac{\partial}{\partial
z}\right)^a\left(\frac{\partial}{\partial\z}\right)^bu^2\right|}$$
both have limit $0$ as $n\to\infty$.

  The following estimates for the derivatives assume $n$ is
  sufficiently large compared to $k$.  The derivatives of
  $u^1=F(n)z^{p(n)}$ (\ref{eq38}) are easy:
  \begin{eqnarray}
    \left|\left(\frac{\partial}{\partial
    z}\right)^a\left(\frac{\partial}{\partial\z}\right)^bu^1\right|&\le&\left|F(n)p(n)(p(n)-1)\cdots(p(n)-k+1)z^{p(n)-k}\right|\nonumber\\
    &\le&F(n)(p(n))^kr_n^{p(n)-k}\label{eq7}
  \end{eqnarray}

  The derivatives of $u^2(z)$ (\ref{eq39}) are considered one term at
  a time.  For the first term:
  \begin{eqnarray*}
    &&\left(\frac{\partial}{\partial
    z}\right)^a\left(\frac{\partial}{\partial\z}\right)^b\left(\chi_n(z)F(n-1)z^{p(n-1)}\right)\\
    &=&F(n-1)\left(\frac{\partial}{\partial
    z}\right)^a\left(\left(\left(\frac{\partial}{\partial\z}\right)^b\chi_n(z)\right)\cdot
    z^{p(n-1)}\right)\\
    &=&F(n-1)\sum_{a_1+a_2=a}^{2^a}\left(\left(\frac{\partial}{\partial
    z}\right)^{a_1}\left(\frac{\partial}{\partial\z}\right)^b\chi_n(z)\right)\left(\left(\frac{\partial}{\partial
    z}\right)^{a_2}z^{p(n-1)}\right)
  \end{eqnarray*}
  The sum is over the $2^a$ terms (with many repeated) that result
  from applying the product rule $a$ times.

  By Lemma \ref{lem1.1},
  $$\left|\left(\frac{\partial}{\partial
    z}\right)^{a_1}\left(\frac{\partial}{\partial\z}\right)^b\chi_n(z)\right|\le\frac{m_{a_1b}}{|z|^{2(a_1+b)}(\Delta
    r_n)^{a_1+b}},$$ and 
  \begin{eqnarray*}
    \left|\left(\frac{\partial}{\partial
    z}\right)^{a_2}z^{p(n-1)}\right|&=&p(n-1)\cdots(p(n-1)-a_2+1)|z|^{p(n-1)-a_2}\\
    &\le&(p(n-1))^{a_2}|z|^{p(n-1)-a_2}.
  \end{eqnarray*}
  Let $\displaystyle{m_k=\max_{a+b\le k}m_{ab}}$.  Then
  \begin{eqnarray}
    &&\left|\left(\frac{\partial}{\partial
    z}\right)^a\left(\frac{\partial}{\partial\z}\right)^b\left(\chi_n(z)F(n-1)z^{p(n-1)}\right)\right|\nonumber\\
    &\le&F(n-1)2^a\max_{a_1\le a}\left\{\frac{m_{a_1b}}{|z|^{2(a_1+b)}(\Delta r_n)^{a_1+b}}\right\}\max_{a_2\le a}\left\{(p(n-1))^{a_2}|z|^{p(n-1)-a_2}\right\}\nonumber\\
    &\le&F(n-1)2^a\frac{m_k}{|z|^{2k}(\Delta r_n)^k}(p(n-1))^a|z|^{p(n-1)-a}\nonumber\\
    &\le&F(n-1)2^k\frac{m_k}{(\Delta r_n)^k}(p(n-1))^k|z|^{p(n-1)-3k}\nonumber\\
    &\le&2^km_k\frac{F(n-1)(p(n-1))^kr_n^{p(n-1)-3k}}{(\Delta r_n)^k}.\label{eq8}
  \end{eqnarray}
  Similarly for the second term of $u^2(z)$,
  \begin{eqnarray}
    &&\left|\left(\frac{\partial}{\partial
    z}\right)^a\left(\frac{\partial}{\partial\z}\right)^b\left((1-\chi_n(z))F(n+1)z^{p(n+1)}\right)\right|\nonumber\\
    &\le&2^km_k\frac{F(n+1)(p(n+1))^kr_n^{p(n+1)-3k}}{(\Delta r_n)^k}.\label{eq9}
  \end{eqnarray}
So, the criteria for all the derivatives of ${\bf u}$ to vanish at the
origin are that the expressions (\ref{eq7}), (\ref{eq8}), and
(\ref{eq9}) must all have limit $0$ as $n\to\infty$.  The hypothesis
(\ref{eq42}) is equivalent to (\ref{eq9})$\to0$.  Comparing
(\ref{eq9}) to (\ref{eq7}) by shifting the index in (\ref{eq7}) from
$n$ to $n+1$, this scalar multiple of (\ref{eq9}) is much larger:
$$\frac{F(n+1)(p(n+1))^kr_n^{p(n+1)-3k}}{(\Delta
  r_n)^k}>F(n+1)(p(n+1))^kr_{n+1}^{p(n+1)-k},$$ so if (\ref{eq9}) has
limit $0$, then so does (\ref{eq7}).

(\ref{eq9})$\to0$ also implies $F(n+1)r_n^{p(n+1)-k}\to0$, which is
enough to show ${\bf u}$ vanishes to infinite order: $\|z^{-k}{\bf
u}(z)\|\to0$ as $z\to0$.

  Shifting the index in (\ref{eq8}) from $n$ to $n+2$ gives the
  following quantity (\ref{eq43}), which is comparable to (\ref{eq9}):
  \begin{eqnarray}
    &&2^km_k\frac{F(n+1)(p(n+1))^kr_{n+2}^{p(n+1)-3k}}{(\Delta
    r_{n+2})^k}\label{eq43}\\
    &<&2^km_k\frac{F(n+1)(p(n+1))^kr_n^{p(n+1)-4k}}{(\Delta
    r_n/r_n)^k}\cdot\frac{(\Delta r_n/r_n)^k}{(\Delta
    r_{n+2}/r_{n+2})^k},\nonumber
  \end{eqnarray}
  and under the additional hypothesis that $\frac{\Delta
    r_n/r_n}{\Delta r_{n+2}/r_{n+2}}$ is a bounded sequence,
  (\ref{eq42}) also implies (\ref{eq8})$\to0$.
\end{proof}

\subsection{Comparing first derivatives}

We want to choose $F$, $p$, and $r_n$ so that
$\displaystyle{\frac{\|{\bf u}_{\z}\|}{\|{\bf u}_z\|}}$ is small, as
$z\to0$.  For $z\in A_n$, $n$ even (and switching indices if $n$ is
odd), expanding the derivative and using (\ref{eq-1}) gives:
\begin{eqnarray*}
  \|{\bf u}_{\z}\|&=&\left|\frac{\partial}{\partial\z}u^2\right|\\
  &\le&\frac{m_{01}}{\Delta r_n}F(n-1)|z|^{p(n-1)}+\frac{m_{01}}{\Delta
  r_n}F(n+1)|z|^{p(n+1)}\\ &=&\frac{m_{01}}{\Delta
  r_n}|z|^{p(n)-1}\left[F(n-1)|z|^{p(n-1)-p(n)+1}+F(n+1)|z|^{p(n+1)-p(n)+1}\right].
\end{eqnarray*}
Using (\ref{eq46}) and introducing a factor $g(n)>0$, for $z\in A_n$:
\begin{eqnarray}\label{eq2}
  \frac{\|{\bf u}_{\z}\|}{\|{\bf u}_z\|}&\le&\frac{m_{01}r_ng(n)}{\Delta
  r_np(n)}\cdot\frac{\left[F(n-1)r_{n+1}^{p(n-1)-p(n)}+F(n+1)r_n^{p(n+1)-p(n)}\right]}{g(n)F(n)}.
\end{eqnarray}
The first fraction in the product is what we want to make small for
large $n$, depending on $p$ and $\displaystyle{\frac{\Delta
r_n}{r_n}}$.  The second fraction we would like to make bounded,
depending on $F$ and $r_n$, and an arbitrary fudge factor $g$.  The
role of $g$ is to manage the size of $F$ and simplify the calculation
proving boundedness of the second factor, possibly at the expense of
affecting the rate at which the first factor approaches $0$.

\section{Examples}\label{sec2}

\begin{example}\label{ex2.1}
  Rosay's example (\cite{rosay}) has $p(n)=n$, $r_n=2^{-n+1}$, and
  $\frac{\Delta r_n}{r_n}=\frac12$.  Then (\ref{eq2}) becomes:
\begin{eqnarray}
  \frac{\|{\bf u}_{\z}\|}{\|{\bf u}_z\|}&\le&\frac{m_{01}2g(n)}{n}\cdot\frac{\left[F(n-1)2^{n}+F(n+1)2^{-n+1}\right]}{g(n)F(n)}.\label{eq40}
\end{eqnarray}
The choice, as in (\cite{rosay}), $F(n)=2^{n^2/2}$, satisfies the
recursive formula 
\begin{equation}\label{eq45}
  F(n-1)2^n=g(n)F(n),
\end{equation}
with $g(n)=\sqrt{2}$.  This simplifies the second factor of the RHS of
(\ref{eq40}), so it is easily seen to be bounded.  The conclusion is
$\frac{\|{\bf u}_{\z}\|}{\|{\bf u}_z\|}\le\frac{C_1}n$ for $z\in A_n$,
and since $\frac1n\le\frac1{-\log_2|z|}\le\frac1{n-1}$ on $A_n$,
  \begin{equation}\label{eq41}
    \frac{\|{\bf u}_{\z}\|}{\|{\bf u}_z\|}\le\frac{C_1}{-\log_2|z|}
  \end{equation}
for all $z\in D_1\setminus\{0\}$.

To check that ${\bf u}$ is smooth and vanishing to infinite order at the
origin, it is enough to verify the condition of Lemma \ref{lem1.2};
for each fixed $k\ge0$:
$$\lim_{n\to\infty}\frac{2^{(n+1)^2/2}(n+1)^k(2^{-n+1})^{(n+1-4k)}}{(2^{-1})^k}=0.$$
\end{example}

The goal of the next example is to improve upon the order of
vanishing of the ratio (\ref{eq41}).

\begin{example}\label{ex2.3}
  Consider $r_n=\frac{1}{\ln(n+1)}$, so
  $r_1=\frac1{\ln(2)}\approx1.44$, $r_2=\frac1{\ln(3)}$, \ldots.  This
  radius shrinks much more slowly than in Example \ref{ex2.1}.  Since
  $$\lim_{n\to\infty}\frac{1-\frac{\ln(n+1)}{\ln(n+2)}}{1/(n\ln(n+2))}=1,$$
  there are constants $C_2$, $c_2>0$ so that for all $n$:
  \begin{eqnarray}
    \frac{c_2}{n\ln(n+2)}<\frac{\Delta
    r_n}{r_n}&=&1-\frac{\ln(n+1)}{\ln(n+2)}<\frac{C_2}{n\ln(n+2)}.\label{eq44}
  \end{eqnarray}
  Let $p(n)=n^2$; then the
  inequality (\ref{eq2}) becomes:
\begin{eqnarray*}
  \frac{\|{\bf u}_{\z}\|}{\|{\bf u}_z\|}&\le&\frac{m_{01}\cdot(n\ln(n+2))g(n)}{c_2n^2}\\
  &&\
  \cdot\frac{\left[F(n-1)\left[\frac{1}{\ln(n+2)}\right]^{-2n+1}+F(n+1)\left[\frac{1}{\ln(n+1)}\right]^{2n+1}\right]}{g(n)F(n)}.
\end{eqnarray*}
  This motivates, in analogy with the previous Example, this choice of
  $g$ and a recursive formula for $F(n)$ as in (\ref{eq45}):
  \begin{eqnarray*}
    g(n)&=&\ln(n+2),\\ F(n)&=&(\ln(n+2))^{[2n-2]}F(n-1)\\
    &=&(\ln(n+2))^{[2n-2]}\cdot(\ln(n+1))^{[2(n-1)-2]}\cdots(\ln(5))^4\cdot(\ln(4))^2\\
    &=&\left[(\ln(n+2))^{n-1}\cdot
    (\ln(n+1))^{n-2}\cdots(\ln(5))^2\cdot(\ln(4))^1\right]^{2}.
  \end{eqnarray*}
  Then the following sequence of ratios is bounded above because it is
  convergent as $n\to\infty$:
  \begin{eqnarray*}
    &&\frac{F(n-1)(\ln(n+2))^{[2n-1]}+F(n+1)(\ln(n+1))^{-[2n+1]}}{g(n)F(n)}\\
    &=&\frac{F(n-1)(\ln(n+2))^{[2n-1]}+\frac{(\ln(n+3))^{2n}(\ln(n+2))^{[2n-2]}F(n-1)}{(\ln(n+1))^{[2n+1]}}}{\ln(n+2)(\ln(n+2))^{[2n-2]}F(n-1)}\\
    &=&1+\frac{(\ln(n+3))^{2n}}{\ln(n+2)(\ln(n+1))^{[2n+1]}}\le C_3.
  \end{eqnarray*}
  Here we used the elementary calculus lemma that
  $\left(\frac{\ln(n+2)}{\ln(n)}\right)^n$ is a bounded sequence.

   The estimate for the ratio of derivatives on $A_n$, for even $n$,
  becomes:
  \begin{eqnarray}
    \frac{\|{\bf u}_{\z}\|}{\|{\bf
    u}_z\|}&\le&\frac{|u_{\z}^2|}{|u_z^1|}\le\frac{m_{01}\cdot(n\ln(n+2))\ln(n+2)}{c_2n^2}\cdot
    C_3\nonumber\\
    &=&\frac{m_{01}(\ln(n+2))^2C_3}{c_2n}=\frac{C_4(\ln(n+2))^2}{n}\label{eq-3}\\
    &<&\frac{C_5(\ln(n+1))^2}{n+2}.\label{eq18}
  \end{eqnarray}
  For $z\in A_n$,
  \begin{eqnarray*}
  &&\frac1{\ln(n+2)}\le|z|\le\frac1{\ln(n+1)}\\
  &\iff&(\ln(n+1))^2\le\frac1{|z|^2}\le(\ln(n+2))^2\\
  &\iff&\frac1{n+2}\le\exp(-\frac1{|z|})\le\frac1{n+1},
  \end{eqnarray*}
  so 
   \begin{eqnarray*}
    \frac{\|{\bf u}_{\z}\|}{\|{\bf u}_z\|}&\le&C_5\frac1{|z|^2\exp(\frac1{|z|})},
  \end{eqnarray*}
  for all $z\in D_{r_1}\setminus\{0\}$.  It remains to check that
  ${\bf u}$ is smooth, and vanishes to infinite order.  The hypothesis
  on $\Delta r_n$ of Lemma \ref{lem1.2} is satisfied (using
  (\ref{eq44})), so for fixed $k$, consider the expression:
  \begin{eqnarray*}
    &&\frac{F(n+1)(p(n+1))^kr_n^{p(n+1)-4k}}{(\Delta r_n/r_n)^k}\\
    &=&\frac{\left[(\ln(n+3))^n(\ln(n+2))^{n-1}\cdots\ln(4)\right]^{2}((n+1)^2)^k\left(\frac1{\ln(n+1)}\right)^{(n+1)^2-4k}}{\left(1-\frac{\ln(n+1)}{\ln(n+2)}\right)^k}\\
    &<&\frac{\left[(\ln(n+3))^n(\ln(n+2))^{n-1}\cdots\ln(4)\right]^{2}(n+1)^{2k}}{(\ln(n+1))^{(n+1)^2-4k}\left(\frac{c_2}{n\ln(n+2)}\right)^k}\\
    &<&\frac{\left[(\ln(n+3))^n(\ln(n+2))^{n-1}\cdots\ln(4)\right]^{2}(n+1)^{3k}(\ln(n+2))^k}{c_2^k(\ln(n+1))^{(n+1)^2-4k}}.
  \end{eqnarray*}
  The last expression has limit zero by the Ratio Test:
  \begin{eqnarray*}
    &&\frac{\frac{\left((\ln(n+4))^{n+1}(\ln(n+3))^n\cdots\ln(4)\right)^2}{c_2^k(\ln(n+2))^{(n+2)^2-4k}}\cdot(n+2)^{3k}(\ln(n+3))^k}{\frac{\left((\ln(n+3))^n\cdots\ln(4)\right)^2}{c_2^k(\ln(n+1))^{(n+1)^2-4k}}\cdot(n+1)^{3k}(\ln(n+2))^k}\\
  &=&\frac{(\ln(n+4))^{2n+2}(\ln(n+1))^{(n+1)^2-4k}(n+2)^{3k}(\ln(n+3))^k}{(\ln(n+2))^{(n+2)^2-4k}(n+1)^{3k}(\ln(n+2))^k}\\
  &<&\frac{(\ln(n+4))^{2n+2+k}(n+2)^{3k}}{(\ln(n+2))^{2n+3+k}(n+1)^{3k}}\to0,
  \end{eqnarray*}
  again using the boundedness of
  $\left(\frac{\ln(n+2)}{\ln(n)}\right)^n$.
\end{example}

\section{A Beltrami-type system}\label{sec3}

Any smooth map ${\bf u}:\co\to\co^2$, ${\bf
  u}=\left[\begin{array}{c}u^1\\u^2\end{array}\right]$, satisfies the
following Beltrami-type system of first-order differential equations
at points where ${\bf u}_z\ne{\bf0}$:
\begin{eqnarray}
  \left[\begin{array}{c}u^1_{\z}\\u^2_{\z}\end{array}\right]&=&\left[\begin{array}{c}u^1_{\z}\\u^2_{\z}\end{array}\right]\frac{[\overline{u^1_z}\ \ \overline{u^2_z}]}{|u^1_z|^2+|u^2_z|^2}\left[\begin{array}{c}u^1_{z}\\u^2_{z}\end{array}\right]\nonumber\\
  &=&\frac{1}{\|{\bf u}_z\|^2}\left[\begin{array}{cc}u^1_{\z}\overline{u^1_z}&u^1_{\z}\overline{u^2_z}\\u^2_{\z}\overline{u^1_z}&u^2_{\z}\overline{u^2_z}\end{array}\right]\left[\begin{array}{c}u^1_{z}\\u^2_{z}\end{array}\right]\label{eq22}\\
  &=&{\bf Q}_{2\times2}(z)\left[\begin{array}{c}u^1_{z}\\u^2_{z}\end{array}\right].\nonumber
\end{eqnarray}
For ${\bf u}$ constructed as in Section \ref{sec1}, on the annuli
$A_n$ with even $n$, $u^1_{\z}\equiv0$, so the first row of the matrix
in (\ref{eq22}) is $[0 0]$, and similarly the second row is $[0 0]$
for odd $n$.  Define ${\bf Q}(0)$ to be the zero matrix.

For a matrix ${\bf Q}(z)$ defined as in (\ref{eq22}) by some fixed
function ${\bf u}$, the operator
$\LL=\displaystyle{\frac{\partial}{\partial\z}-{\bf
Q}(z)\frac{\partial}{\partial z}}$ is complex linear.  If, on some
neighborhood of $z=0$, the ${\bf Q}(z)$ entries are defined and small
enough, then $\LL$ is elliptic (in the sense of \cite{aim} Section
7.4).

In the following Theorem, we consider ${\bf Q}(z)$ for the example
${\bf u}(z)$ from Example \ref{ex2.3}.  If we restrict ${\bf u}$ and
${\bf Q}$ to $z$ in some sufficiently small neighborhood of the
origin, ${\bf u}$ will be a solution of the elliptic equation
$\LL({\bf u})=0$.

\begin{thm}\label{thm4.1}
  For ${\bf u}$ as in Example \ref{ex2.3}, let
  $q_{ij}(z)=\displaystyle{\frac{u^i_{\z}\overline{u^j_z}}{\|{\bf
        u}_z\|^2}}$ denote the $i,j$ entry in the matrix ${\bf Q}(z)$
  from {\rm{(\ref{eq22})}}.
  \begin{itemize}
    \item
      $q_{ij}\in\cc^\infty(D_{r_1}\setminus\{0\})\cap\cc^0(D_{r_1})$;
    \item $q_{ij}$ vanishes to infinite order:
      $|z^{-k}q_{ij}(z)|\to0$ as $z\to0$ for any $k\ge0$;
    \item The partial derivatives exist at the origin:
      $\frac{\partial}{\partial x}q_{ij}(0)=\frac{\partial}{\partial
      y}q_{ij}(0)=0$;
    \item For any $0<r<r_1$, $q_{22}$ does not have the Lipschitz
      property on $D_r$.
  \end{itemize}
\end{thm}
\begin{proof}
  The $\cc^\infty$ claim follows from the smoothness of ${\bf u}$ on
  $D_{r_1}$ and the nonvanishing of ${\bf u}_z$ for $z\ne0$.

The sum $\displaystyle{|q_{11}|^2+|q_{12}|^2+|q_{21}|^2+|q_{22}|^2}$
is exactly $\displaystyle{\frac{\|{\bf u}_{\z}\|^2}{\|{\bf
      u}_z\|^2}}$, which vanishes to infinite order as $z\to0$ for
${\bf u}$ as in Example \ref{ex2.3}.  It follows that each $q_{ij}$
also vanishes to infinite order, which implies ${\bf Q}$ and the
entries $q_{ij}$ are continuous at the origin, with the previously
assigned values $q_{ij}(0)=0$.  The flatness also implies the
existence of all directional derivatives at the origin of $\co=\re^2$;
for the $x$ direction,
  $$\left.\frac{\partial}{\partial
  x}q_{ij}\right]_{(x,y)=(0,0)}=\lim_{x\to0}\frac{q_{ij}(x,0)-q_{ij}(0,0)}{x}=0.$$

  The last claim takes up the rest of the Proof; the plan is to show
  there is a sequence of points $x_n\in\co$ approaching $0$ so that
  $\displaystyle{\left.\frac{\partial}{\partial\z}\frac{u^2_{\z}\overline{u^2_z}}{\|{\bf
        u}_z\|^2}\right]_{z=x_n}}$ is an unbounded sequence.  If
    $q_{22}$ had a Lipschitz property on $D_r$
    ($|q_{22}(z_1)-q_{22}(z_2)|\le K|z_1-z_2|$ for some $K$ and all
    $z_1$, $z_2$), then its derivatives would be bounded; the
    unboundedness of the derivative also directly shows
    $q_{22}\notin\cc^1(D_r)$.

    It is enough, and simpler, to consider only $n$ which are even and
    sufficiently large.  This will involve some estimates for
    derivatives that are more precise than (\ref{eq8}).

  We choose the sequence $x_n=r_{n+1}+\frac12\Delta r_n+0i\in A_n$;
  then by construction of $s$ and $\chi_n$, $\chi_n(x_n)=\frac12$, and
  (\ref{eq0}) gives
  $\frac{\partial\chi_n}{\partial\z}(x_n)=\frac{\partial\chi_n}{\partial
  z}(x_n)=\frac1{\Delta r_n}$.  From (\ref{eq-2}),
  $$\frac{\partial^2\chi_n}{\partial
  z\partial\z}(x_n)=-\frac{\partial^2\chi_n}{\partial
  z^2}(x_n)=-\frac{\partial^2\chi_n}{\partial\z^2}(x_n)=\frac{1}{2x_n\Delta
  r_n}$$ (this is where we use the $s^{\prime\prime}(\frac12)=0$
  assumption, to simplify the calculation).

  For $r_n$ as in Example \ref{ex2.3},
  $\frac1{\ln(n+2)}<x_n<\frac1{\ln(n+1)}$, and $\Delta
  r_n=\frac1{\ln(n+1)}-\frac1{\ln(n+2)}$ satisfies
  $$0<c_6n(\ln(n+2))^2<\frac1{\Delta r_n}<C_6n(\ln(n+2))^2.$$

  In the following expression,
  \begin{eqnarray}
    &&\frac{\partial}{\partial\z}\frac{u^2_{\z}\overline{u^2_z}}{\|{\bf
  u}_z\|^2}=\frac{u_{\z\z}^2\overline{u_z^2}}{\|{\bf
  u}_z\|^2}+\frac{u_{\z}^2\overline{u_{zz}^2}}{\|{\bf
  u}_z\|^2}-\frac{u_{\z}^2\overline{u_z^2}\frac{\partial}{\partial\z}(u_z^1\overline{u_z^1}+u_z^2\overline{u_z^2})}{\|{\bf
  u}_z\|^4}\nonumber\\ &=&\frac{u_{\z\z}^2\overline{u_z^2}}{\|{\bf
  u}_z\|^2}+\frac{u_{\z}^2\overline{u_{zz}^2}(u_z^1\overline{u_z^1}+u_z^2\overline{u_z^2})-u_{\z}^2\overline{u_z^2}(u_z^1\overline{u_{zz}^1}+u_{z\z}^2\overline{u_z^2}+u_z^2\overline{u_{zz}^2})}{\|{\bf
  u}_z\|^4},\label{eq30}
  \end{eqnarray}
  the terms with the largest magnitude are the ones involving the
  second $z$-derivatives, $u_{zz}^1$ and $u_{zz}^2$.  Evaluated at
  points in the sequence $x_n$, these terms individually grow at least
  as fast as some constant multiple of $n$.  However, due to some
  cancellations, the overall growth rate turns out to be less than
  $n$; the Theorem will be proved by showing one of the terms is
  unbounded and the remaining terms have a slower rate of growth.

  The first cancellation is that the second and last terms in the
  numerator of the second fraction in (\ref{eq30}) are exactly
  opposites.  This leaves two terms with $zz$-derivatives; using the
  power rule on $u^1=F(n)z^{n^2}$ in the interior of $A_n$ gives:
  $$u_{zz}^1=u_z^1\cdot\frac{n^2-1}{z}.$$ (\ref{eq30}) becomes:
  \begin{equation}\label{eq34}
    \frac{u_{\z\z}^2\overline{u_z^2}}{\|{\bf
  u}_z\|^2}+\frac{u_{\z}^2u_z^1\overline{u_z^1}\left(\overline{u_{zz}^2-\frac{(n^2-1)u_z^2}{z}}\right)}{\|{\bf
  u}_z\|^4}-\frac{u_{\z}^2\overline{u_z^2}u_{z\z}^2\overline{u_z^2}}{\|{\bf
  u}_z\|^4}.
  \end{equation}
  We will show that the last of the three terms in (\ref{eq34}) is the
  dominant one.

  First, consider the ratio:
  \begin{eqnarray}
    &&|u_z^2|/|u_z^1|\nonumber\\
    &=&\left|F(n-1)(\frac{\partial\chi_n}{\partial
    z}z^{(n-1)^2}+\chi_n\cdot(n-1)^2z^{(n-1)^2-1})\right.\label{eq28}\\
    &&\ +\left.F(n+1)(-\frac{\partial\chi_n}{\partial
    z}z^{(n+1)^2}+(1-\chi_n)(n+1)^2z^{n^2+2n})\right|/\left|F(n)n^2z^{n^2-1}\right|.\nonumber
  \end{eqnarray}
  The following calculation for $z\in A_n$ is similar to that for the
  estimate (\ref{eq-3}).
  \begin{eqnarray*}
    \frac{|u_z^2|}{|u_z^1|}&\le&\frac{F(n-1)\left[\left|\frac{\partial\chi_n}{\partial
    z}\right||z|^{(n-1)^2}+\chi_n\cdot(n-1)^2|z|^{n^2-2n}\right]}{F(n)n^2|z|^{n^2-1}}\\\
    &&\ +\frac{F(n+1)\left[\left|\frac{\partial\chi_n}{\partial
    z}\right||z|^{(n+1)^2}+(1-\chi_n)(n+1)^2|z|^{n^2+2n}\right]}{F(n)n^2|z|^{n^2-1}}
  \end{eqnarray*}
  From Example \ref{ex2.3}, recalling $F(n)=(\ln(n+2))^{2n-2}F(n-1)$,
  $\frac1{\ln(n+2)}\le|z|\le\frac1{\ln(n+1)}$, and
  $\left|\frac{\partial\chi_n}{\partial
  z}\right|\le\frac{m_{01}}{\Delta r_n}$, it follows that there is
  some constant $C_7>1$ so that
  \begin{equation}\label{eq27}
    \frac{|u_z^2|}{|u_z^1|}\le C_7\ln(n+2).
  \end{equation}

  To estimate the denominators of (\ref{eq34}),
  \begin{eqnarray*}
    \frac{\|{\bf
    u}_z\|^2}{|u_z^1|^2}&=&1+\left(\frac{|u_z^2|}{|u_z^1|}\right)^2\le1+\left(C_7\ln(n+2)\right)^2\\
  \implies\|{\bf u}_z\|^2&\le&C_8(\ln(n+2))^2|u_z^1|^2.
  \end{eqnarray*}
  So we get a lower bound for the third term in (\ref{eq34}),
  \begin{eqnarray}
    \frac{\left|u_{\z}^2\overline{u_z^2}u_{z\z}^2\overline{u_z^2}\right|}{\|{\bf
    u}_z\|^4}&\ge&\frac{\left|u_{\z}^2\right|\left|u_z^2\right|^2\left|u_{z\z}^2\right|}{C_8^2(\ln(n+2))^4|u_z^1|^4},\label{eq25}
  \end{eqnarray}
  and consider (\ref{eq25}) one factor at a time.
  \begin{eqnarray}
  \frac{|u_{\z}^2|}{|u_z^1|}&=&\frac{\left|\frac{\partial\chi_n}{\partial\z}\cdot F(n-1)z^{(n-1)^2}-\frac{\partial\chi_n}{\partial\z}\cdot F(n+1)z^{(n+1)^2}\right|}{\left|F(n)n^2z^{n^2-1}\right|}\label{eq29}\\
  \left.\frac{|u_{\z}^2|}{|u_z^1|}\right]_{z=x_n}&=&\frac1{\Delta r_n}\left|\frac1{(\ln(n+2))^{2n-2}n^2x_n^{2n-2}}-\frac{(\ln(n+3))^{2n}x_n^{2n+2}}{n^2}\right|\nonumber\\
  &\ge&c_6n(\ln(n+2))^2\left(\frac{(\ln(n+1))^{2n-2}}{(\ln(n+2))^{2n-2}n^2}-\frac{(\ln(n+3))^{2n}}{n^2(\ln(n+1))^{2n+2}}\right)\nonumber\\
  &\ge&\frac{c_4(\ln(n+2))^2}n.\label{eq26}
  \end{eqnarray}
  This lower bound (\ref{eq26}) is comparable to the upper bound
  (\ref{eq-3}), which is used in the next step to find a lower bound
  for $\frac{|u_z^2|}{|u_z^1|}$ (\ref{eq28}).  Note that the
  expression (\ref{eq28}) has two terms which, using the equality of
  $\frac{\partial\chi_n}{\partial\z}$ and
  $\frac{\partial\chi_n}{\partial z}$ when evaluated at $x_n$, match
  two of the terms from $u_{\z}^2$ in (\ref{eq29}):
  \begin{eqnarray}
  \left.\frac{|u_{z}^2|}{|u_z^1|}\right]_{z=x_n}&=&\left|\left.\frac{u_{\z}^2}{u_z^1}\right]_{x_n}+\frac{(n-1)^2}{2(\ln(n+2))^{2n-2}n^2x_n^{2n-1}}\right.\nonumber\\
  &&\
  \left.+\frac{(\ln(n+3))^{2n}(n+1)^2x_n^{2n+1}}{2n^2}\right|\nonumber\\
  &\ge&\frac{(n-1)^2(\ln(n+1))^{2n-1}}{2(\ln(n+2))^{2n-2}n^2}-\left|\left.\frac{u_{\z}^2}{u_z^1}\right]_{x_n}\right|\nonumber\\
  &&\ -\frac{(\ln(n+3))^{2n}(n+1)^2}{2n^2(\ln(n+1))^{2n+1}}\nonumber\\
  &\ge&c_9\ln(n+2)-C_4\frac{(\ln(n+2))^2}n-C_{10}\frac1{\ln(n+1)}\nonumber\\
  &\ge&c_7\ln(n+2).\label{eq31}
  \end{eqnarray}
This lower bound is comparable to the upper bound (\ref{eq27}).  The
remaining factor from (\ref{eq25}) involves second derivatives:
\begin{eqnarray*}
  &&|u_{z\z}^2|/|u_z^1|\\
  &=&\left|F(n-1)\left(\frac{\partial^2\chi_n}{\partial
  z\partial\z}\cdot
  z^{(n-1)^2}+2\frac{\partial\chi_n}{\partial\z}\cdot(n-1)^2z^{n^2-2n}\right)\right.\\
  &&\ \left.-F(n+1)\left(\frac{\partial^2\chi_n}{\partial
  z\partial\z}\cdot
  z^{(n+1)^2}+\frac{\partial\chi_n}{\partial\z}\cdot(n+1)^2z^{n^2+2n}\right)\right|/|F(n)n^2z^{n^2-1}|\\
  &\ge&\frac{F(n-1)}{F(n)n^2}\left(\left|\frac{\partial\chi_n}{\partial\z}\right|(n-1)^2|z|^{-2n+1}-\left|\frac{\partial^2\chi_n}{\partial z\partial\z}\right||z|^{-2n+2}\right)\\
  &&\ -\frac{F(n+1)}{F(n)n^2}\left(\left|\frac{\partial^2\chi_n}{\partial z\partial\z}\right||z|^{2n+2}+\left|\frac{\partial\chi_n}{\partial\z}\right|(n+1)^2|z|^{2n+1}\right).
\end{eqnarray*}
Evaluating at $x_n$ (again, for sufficiently large $n$),
\begin{eqnarray}
  \left.\frac{|u_{z\z}^2|}{|u_z^1|}\right]_{x_n}&\ge&\frac{1}{(\ln(n+2))^{2n-2}n^2}\left(\frac{(n-1)^2}{\Delta r_nx_n^{2n-1}}-\frac{1}{2x_n\Delta r_n}\frac1{x_n^{2n-2}}\right)\nonumber\\
  &&-\frac{(\ln(n+3))^{2n}}{n^2}\left(\frac{x_n^{2n+2}}{2x_n\Delta r_n}+\frac1{\Delta r_n}(n+1)^2x_n^{2n+1}\right)\nonumber\\
  &\ge&\frac{1}{(\ln(n+2))^{2n-2}n^2}\left(c_6n(\ln(n+2))^2(n-1)^2(\ln(n+1))^{2n-1}\right.\nonumber\\
  &&\ \ \ \ \left.-\frac12C_6n(\ln(n+2))^2(\ln(n+2))^{2n-1}\right)\nonumber\\
  &&-\frac{(\ln(n+3))^{2n}C_6(\ln(n+2))^2}{n^2}\left(\frac1{2(\ln(n+1))^{2n+1}}\right.\nonumber\\
  &&\ \ \ \ \left.+\frac{(n+1)^2}{(\ln(n+1))^{2n+1}}\right)\nonumber\\
  &\ge&c_{11}n(\ln(n+2))^3.\label{eq33}
\end{eqnarray}
So, the term from (\ref{eq25}) is bounded below by a product including
factors from (\ref{eq26}), (\ref{eq31}), and (\ref{eq33}):
  \begin{eqnarray}
    \left.\frac{\left|u_{\z}^2(\overline{u_z^2})^2u_{z\z}^2\right|}{\|{\bf
    u}_z\|^4}\right]_{x_n}&\ge&\frac{c_4(\ln(n+2))^2(c_7\ln(n+2))^2c_{11}n(\ln(n+2))^3}{nC_8^2(\ln(n+2))^4}\nonumber\\
  &\ge&c_{12}(\ln(n+2))^3.\label{eq32}
  \end{eqnarray}

From the second term in (\ref{eq34}), we consider the following quantity:
\begin{eqnarray*}
  &&u_{zz}^2-\frac{(n^2-1)u_z^2}{z}\\
  &=&F(n-1)\left(\frac{\partial^2\chi_n}{\partial
  z^2}z^{(n-1)^2}+2\frac{\partial\chi_n}{\partial
  z}(n-1)^2z^{n^2-2n}\right.\\ &&\
  \left.+\chi_n\cdot(n-1)^2(n^2-2n)z^{n^2-2n-1}\right)\\
  &&+F(n+1)\left(-\frac{\partial^2\chi_n}{\partial
  z^2}z^{(n+1)^2}-2\frac{\partial\chi_n}{\partial
  z}(n+1)^2z^{n^2+2n}\right.\\ &&\
  \left.+(1-\chi_n)\cdot(n+1)^2(n^2+2n)z^{n^2+2n-1}\right)\\
  &&-F(n-1)(n^2-1)\left(\frac{\partial\chi_n}{\partial
  z}z^{n^2-2n}+\chi_n\cdot(n-1)^2z^{n^2-2n-1}\right)\\
  &&-F(n+1)(n^2-1)\left(-\frac{\partial\chi_n}{\partial
  z}z^{n^2+2n}+(1-\chi_n)\cdot(n+1)^2z^{n^2+2n-1}\right).
\end{eqnarray*}
The cancellation of the $n^4$ quantities is the key step.  The ratio
\begin{eqnarray*}
  &&\left|u_{zz}^2-\frac{(n^2-1)u_z^2}{z}\right|/|u_z^1|\\
  &=&\left|F(n-1)\left(\frac{\partial^2\chi_n}{\partial z^2}z^{(n-1)^2}+\frac{\partial\chi_n}{\partial z}(n^2-4n+3)z^{n^2-2n}\right.\right.\\
  &&\ \left.-\chi_n\cdot(n-1)^2(2n-1)z^{n^2-2n-1}\right)\\
  &&+F(n+1)\left(\frac{\partial^2\chi_n}{\partial z^2}z^{(n+1)^2}-\frac{\partial\chi_n}{\partial z}(n^2+4n+3)z^{n^2+2n}\right.\\
  &&\ \left.\left.+(1-\chi_n)(n+1)^2(2n+1)z^{n^2+2n-1}\right)\right|/\left|F(n)n^2z^{n^2-1}\right|
\end{eqnarray*}
has an upper bound on the $x_n$ sequence:
\begin{eqnarray}
  &&\left.\frac{\left|u_{zz}^2-\frac{(n^2-1)u_z^2}{z}\right|}{|u_z^1|}\right]_{x_n}\nonumber\\
  &\le&\frac1{(\ln(n+2))^{2n-2}n^2}\left(\frac1{2x_n\Delta r_nx_n^{2n-2}}+\frac{n^2-4n+3}{\Delta r_nx_n^{2n-1}}+\frac{(n-1)^2(2n-1)}{2x_n^{2n}}\right)\nonumber\\
  &&+\frac{(\ln(n+3))^{2n}}{n^2}\left(\frac{x_n^{2n+2}}{2x_n\Delta r_n}+\frac{(n^2+4n+3)x_n^{2n+1}}{\Delta r_n}+\frac{(n+1)^2(2n+1)x_n^{2n}}2\right)\nonumber\\
  &\le&\frac1{(\ln(n+2))^{2n-2}n^2}\left(C_6n(\ln(n+2))^2(\frac12+n^2-4n+3)(\ln(n+2))^{2n-1}\right.\nonumber\\
  &&\ \left.+\frac12(n-1)^2(2n-1)(\ln(n+2))^{2n}\right)\nonumber\\
  &&+\frac{(\ln(n+3))^{2n}}{n^2}\left(C_6n(\ln(n+2))^2(\frac12+n^2-4n+3)\frac1{(\ln(n+1))^{2n+1}}\right.\nonumber\\
  &&\ \left.+\frac{(n+1)^2(2n+1)}{(\ln(n+1))^{2n}}\right)\nonumber\\
  &\le&C_{13}n(\ln(n+2))^3.\label{eq36}
\end{eqnarray}
The middle term from (\ref{eq34}), evaluated at points $x_n$, has
factors bounded by (\ref{eq-3}), (\ref{eq31}), and (\ref{eq36}):
  \begin{eqnarray}
    &&\left.\left|\frac{u_{\z}^2u_z^1\overline{u_z^1}\left(\overline{u_{zz}^2-(n^2-1)u_z^2/z}\right)}{\|{\bf
    u}\|^4}\right|\right]_{x_n}\nonumber\\
    &\le&\left.\frac{|u_{\z}^2||u_z^1|^2\left|u_{zz}^2-(n^2-1)u_z^2/z\right|}{|u_z^2|^4}\right]_{x_n}\nonumber\\
    &=&\left.\frac{|u_{\z}^2|}{|u_z^1|}\right]_{x_n}\left.\frac{|u_z^1|^4}{|u_z^2|^4}\right]_{x_n}\left.\frac{\left|u_{zz}^2-(n^2-1)u_z^2/z\right|}{|u_z^1|}\right]_{x_n}\nonumber\\
    &\le&C_4\frac{(\ln(n+2))^2}{n}\frac{1}{(c_7\ln(n+2))^4}C_{13}n(\ln(n+2))^3\nonumber\\
    &\le&C_{14}\ln(n+2).\label{eq37}
  \end{eqnarray}

The first term from (\ref{eq34}) involves the second $\z$-derivative:
\begin{eqnarray*}
  \frac{|u_{\z\z}^2|}{|u_z^1|}&=&\frac{\left|\frac{\partial^2\chi_n}{\partial\z^2}F(n-1)z^{(n-1)^2}-\frac{\partial^2\chi_n}{\partial\z^2}F(n+1)z^{(n+1)^2}\right|}{\left|F(n)n^2z^{n^2-1}\right|}\\ &\le&\frac{F(n-1)}{F(n)n^2}\left|\frac{\partial^2\chi_n}{\partial\z^2}\right||z|^{-2n+2}+\frac{F(n+1)}{F(n)n^2}\left|\frac{\partial^2\chi_n}{\partial\z^2}\right||z|^{2n+2}.\\ \left.\frac{|u_{\z\z}^2|}{|u_z^1|}\right]_{x_n}&\le&\frac1{2x_n\Delta
      r_nn^2}\left(\frac1{(\ln(n+2))^{2n-2}}\cdot\frac1{x_n^{2n-2}}+(\ln(n+3))^{2n}x_n^{2n+2}\right)\\ &\le&\frac{C_6n(\ln(n+2))^2}{2n^2}\left(\frac{(\ln(n+2))^{2n-1}}{(\ln(n+2))^{2n-2}}+\frac{(\ln(n+3))^{2n}}{(\ln(n+1))^{2n+1}}\right)\\ &\le&C_{15}\frac{(\ln(n+2))^3}{n}.
\end{eqnarray*}
The first term from (\ref{eq34}) also approaches $0$ for large $n$:
\begin{eqnarray*}
  \left.\frac{|u_{\z\z}^2\overline{u_z^2}|}{\|{\bf
      u}_z\|^2}\right]_{x_n}&\le&\left.\frac{|u_{\z\z}^2|}{|u_z^1|}\right]_{x_n}\left.\frac{|u_z^2|}{\|{\bf
          u}_z\|}\right]_{x_n}\le C_{15}\frac{(\ln(n+2))^3}{n}.
\end{eqnarray*}
The conclusion from (\ref{eq30}) and (\ref{eq34}) is:
 \begin{eqnarray*}
    \left.\left|\frac{\partial}{\partial\z}\frac{u^2_{\z}\overline{u^2_z}}{\|{\bf
  u}_z\|^2}\right|\right]_{x_n}&\ge&c_{12}(\ln(n+2))^3-C_{14}\ln(n+2)-C_{15}\frac{(\ln(n+2))^3}{n}\\
    &>&\frac{c_{16}}{x_n^3}.
  \end{eqnarray*}
\end{proof}

\section{Remarks and Questions}\label{sec4}

\begin{rem}
  The regularity of the coefficients is an important consideration in
  the analysis of unique continuation properties for some PDEs (for
  example, \cite{lnw} for strong UCP, and \cite{ivv} for weak UCP),
  which is why we presented the detailed Proof of Theorem
  \ref{thm4.1}.  However, we do not yet understand the sharpness of
  Example \ref{ex2.3} and Theorem \ref{thm4.1} for this particular
  unique continuation problem; would improved regularity of ${\bf
    Q}(z)$ (in addition to the flatness property) imply a strong
  unique continuation property, or, oppositely, is there some
  counterexample where ${\bf Q}(z)$ is smooth?
\end{rem}

\begin{rem}
  \cite{rosay} shows how Example \ref{ex2.1} can be modified so that
  the origin is a non-isolated zero of ${\bf u}$; it is a matter of
  replacing quantities $z^N$ in (\ref{eq38}), (\ref{eq39}) by
  $z^{N-1}(z-a_n)$ for a sequence $a_n$ and re-working the cutoff
  functions $\chi_n$.  Our Example \ref{ex2.3} can be modified in an
  analogous way but we have not worked out all the details.
\end{rem}

\begin{rem}
  By a construction analogous to (\ref{eq22}), the function ${\bf u}$
  from Example \ref{ex2.3} also satisfies a real linear, elliptic
  equation of the form ${\bf u}_{\z}=\widetilde{\bf
    Q}_{2\times2}\overline{{\bf u}_z}$.  $\widetilde{{\bf Q}}(z)$ is
  not the same as ${\bf Q}(z)$ but also has entries vanishing to
  infinite order.
\end{rem}

\begin{rem}
  Another differential inequality, considered by \cite{rosay}, is
  $\|{\bf u}_{\z}\|\le K\|{\bf u}\|^\alpha\|{\bf u}_z\|$, for
  $0<\alpha<1$.  Our attempts to use the construction of Section
  \ref{sec1} to find smooth functions ${\bf u}$ satisfying the
  inequality and vanishing to infinite order at an isolated zero have
  not yet met any success.  \cite{rosay} proves a weak unique
  continuation property for $\alpha=\frac12$, but the strong property
  remains an open question.
\end{rem}

\end{document}